\documentclass[12pt]{amsart}

\usepackage{amsmath,amssymb,amscd,amsthm,amsxtra,verbatim}
\usepackage{latexsym}
\usepackage{color}

\newtheorem{theorem}{Theorem}
\newtheorem{lemma}{Lemma}
\newtheorem{proposition}{Proposition}
\newtheorem{remark}{Remark}

\newtheorem{definition}{Definition}
\newtheorem{corollary}{Corollary}

\newcommand{\q}{\quad}
\newcommand{\qq}{\quad\quad}
\newcommand{\qqq}{\quad\quad\quad}

\newcommand{\nf}{\infty}

\newcommand{\al}{\alpha}
\newcommand{\be}{\beta}
\newcommand{\ga}{\gamma}
\newcommand{\Ga}{\Gamma}
\newcommand{\de}{\delta}

\newcommand{\De}{\Delta}
\newcommand{\ve}{\varepsilon}

\newcommand{\la}{\lambda}

\newcommand{\vp}{\varphi}
\newcommand{\om}{\omega}
\newcommand{\Om}{\Omega}

\newcommand{\rn}{{\mathbf R}^n}

\newcommand{\zp}{\mathbf Z^+}

\newcommand{\lp}{L^{p}}

\newcommand{\wlp}{L^{p,\infty}}

\newcommand{\lprn}{L^{p}(\rn)}

\newcommand{\lo}{L^{1}}

\newcommand{\li}{L^{\infty}}

\newcommand{\cs}{\mathcal S}

\newcommand{\cm}{\mathcal M}

\newcommand{\lab}{\label}

\newcommand{\loc}{\textup{loc}}

\newcommand{\intrn}{\int_{\rn}}

\newcommand{\f}{\frac}

\newcommand{\p}{\partial}

\newcommand{\wt}{\widetilde}
\newcommand{\wh}{\widehat}

\newcommand{\cf}{\mathcal F}
\newcommand{\tf}{\tfrac}
\newcommand{\zzz}{\mathbf Z}
\newcommand{\rar}{\rightarrow}
\newcommand{\mcs}{\mathcal{S}}
\newcommand{\iy}{\infty}
\newcommand{\hi}{H^{p,\iy}}
\newcommand{\bt}{\begin{theorem}}
\newcommand{\et}{\end{theorem}}
\newcommand{\dis}{\displaystyle}
\newcommand{\ep}{\epsilon}
\newcommand{\vph}{\vp}
\newcommand{\cc}{\mathbf{C}}
\newcommand{\eset}{\emptyset}

\begin{document}

\title
[Square function characterization of weak Hardy spaces]
{Square function characterization of weak Hardy spaces}

\author{Danqing He}

\address{
Department of Mathematics\\
University of Missouri\\
Columbia, MO 65211, USA}

\email{dhd27@mail.missouri.edu}


\date{\today}



\begin{abstract} 
We obtain a new square function characterization of the weak Hardy space $\hi$ for all $p\in(0,\iy)$. This space consists of all tempered distributions whose smooth maximal function lies in weak $L^p$.    Our proof  is based on 
 interpolation between $H^p$ spaces. 
 The main difficulty we overcome  is the lack of a good dense subspace of 
$\hi$ which forces us to work with general $\hi$ distributions.
\end{abstract}

\maketitle

\section{Introduction}

In this work we extend Peetre's \cite{P}  characterization of Hardy spaces in terms of the Littlewood-Paley square function   to    the setting of weak Hardy  spaces, $\hi$.

Hardy spaces first appeared in the work of
Hardy \cite{h} in 1914. Their study was based on complex methods and their theory was one-dimensional.
Burkholder, Gundy and Silverstein \cite{bgs} proved that    a complex function  $F=u+iv$ on the upper half space
lies in $ H^p$  if and only if  the nontangential maximal function of
$u$ lies in $L^p(\mathbf R)$. This result inspired the
extension of the theory of Hardy spaces to higher dimensions, in particular the
celebrated work of  Fefferman and  Stein \cite{FS} on this topic.
A deep structural characterization of these spaces was given by
Coifman \cite{C} and Latter \cite{L},  in terms of their
atomic decomposition. The books of Lu \cite{lu},  Uchiyama \cite{U}, and Triebel \cite{triebel1} provide
comprehensive expositions on the theory of Hardy spaces on Euclidean spaces.
The theory of Hardy spaces has proved to be so rich and fruitful that has been extended to 
spaces of homogeneous type; we refer to the works of 
Coifman, and Weiss \cite{CW},
Mac\'{i}as, and Segovia \cite{MS},
Duong, and Yan \cite{DY},
Han, M\"{u}ller and Yang \cite{HMY} and Hu, Yang, Zhou \cite{HYZ} for   results and applications in this setting.

The Hardy-Lorentz spaces $H^{p,q}$, $0<p<\nf$, $0<q\le \nf$ are defined as the spaces of all distributions whose smooth maximal function lies in the Lorentz space $L^{p,q}$. These spaces were studied by
Fefferman and Soria \cite{FSo}, Alvarez \cite{A}, and  Abu-Shammala and Torchinsky \cite{AS-T}.   Fefferman, Riviere and Sagher \cite{FRS} showed that the  $H^{p,q}$ spaces are intermediate spaces of Hardy spaces in the
$K$-interpolation method.  The interpolation result in \cite{FRS} was only proved for Schwartz functions, which is not a dense
subspace of $H^{p,q}$  when $q=\nf$, a fact also  observed in  \cite{FSo}.

In this article we focus on the case $q=\nf$ which presents difficulties due to the lack of a good dense subspace of it. We prove an interpolation theorem for weak Hardy spaces as intermediate spaces of  Hardy spaces and we work with general tempered distributions and the grand maximal function to accomplish this; for this reason our proof looks unavoidably complicated. 
As an application we obtain a new Littlewood-Paley square function characterization of weak Hardy spaces. This shows that $H^{p,\nf}$ is a natural extension of   $L^{p,\nf}$  when $p\le 1$, just like 
 $H^p$ is a natural extension of $L^p$ for $p\le 1$, in view of the 
 Littlewood-Paley theorem on    weak $L^p$.   This characterization reveals the orthogonality of weak $L^p$ spaces for $p<1$ (Corollary \ref{last} ),
which is crucial in the theory of multilinear paraproduct.

We now state our main result. We denote by 
$
\De_j (f)   = \Psi_{2^{-j}}*f
$
the Littlewood-Paley operator of a distribution $f$, where $\Psi_t(x) = t^{-n}\Psi(x/t)$. 

\begin{theorem}\lab{sq}
Let $\Psi$ be a 
radial Schwartz function   on $\rn$ whose Fourier 
transform is nonnegative, supported in $1-\f17\le |\xi|\le 2$, and satisfies $\sum_{j\in\zzz} \wh{\Psi}(2^{-j}\xi)=1$ when $\xi\neq 0$. Let 
$\De_j$ be the Littlewood--Paley 
operators associated with $\Psi$ and let $0<p <\iy$. 
Then there exists a constant 
$C=C_{n,p,\Psi}$ such that for all $f\in H^{p,\nf}(\rn)$ we have 
\begin{equation}\lab{6.4.6.1st-dire}
\Big\|\Big(\sum_{j\in\zzz} |\De_j(f)|^2\Big)^{\!\f12} \Big\|_{L^{p,\nf}}\le 
C \big\|f\big\|_{H^{p,\nf}}\, .
\end{equation}
Conversely, suppose that a  tempered distribution $f$   satisfies 
\begin{equation}\lab{6.4.6.2nd-dire099887}
\Big\|\Big(\sum_{j\in\zzz} |\De_j(f)|^2\Big)^{\!\f12} \Big\|_{L^{p,\nf}}<\nf \, .
\end{equation}
Then there exists a unique polynomial $Q$ such that $f-Q$ 
lies in  
 $H^{p,\nf}$ and satisfies 
\begin{equation}\lab{6.4.6.2nd-dire}
\f{1}{C}\big\|f-Q\big\|_{H^{p,\nf}}\le 
\Big\|\Big(\sum_{j\in\zzz} |\De_j(f)|^2\Big)^{\!\f12} \Big\|_{L^{p,\nf}}
\, .
\end{equation}  
\end{theorem}

The proof of this theorem is based on Theorem \ref{In1} discussed in Section \ref{Int}.

\section{Background}
We introduce the weak Hardy space $H^{p,\infty} $ via the Poisson  maximal function, following the classical definition of the Hardy space. So, we begin our study by listing a result containing the equivalence of quasinorms of several kinds of maximal functions, which also appear in the theory of Hardy spaces.

 We denote by $\ell^2$ the space $\ell^2(\zzz)$ of all square-integrable sequences and by $\ell^2(L)$ the finite-dimensional space of all sequences of 
 length $L\in \zp$ with the $\ell^2 $ norm.
We say that a  sequence of distributions $\{f_j\}_j$ lies in $\mathcal S'(\rn, \ell^2)$ if there are constants $C,M>0$ such that for every $\vp \in \mathcal S(\rn)$ we have 
$$
\big\|\{ \langle f_j, \vp \rangle\}_j \big\|_{\ell^2}= \Big(\sum_{j} |\langle f_j, \vp \rangle |^2 \Big)^{1/2} \le C\, \sum_{|\al|, |\be|\le M} 
 \sup_{y\in \rn } |y^\be \p^\al \vp(y)|.
$$

A  sequence of distributions $\vec f=\{f_j\}_j$ in $\mathcal S'(\rn, \ell^2)$ is called  bounded  if 
$$
\big\|\{   \vp*f_j \}_j \big\|_{\ell^2} =\Big(\sum_{j} |  \vp* f_j    |^2 \Big)^{1/2}  \in L^\nf(\rn)
 $$
for every $\vp$ in $ \mathcal S(\rn)$.

Let $a,b>0$ and let $\Phi$ be a Schwartz function on $\rn$. 
 
\begin{definition}\lab{d6.4.11}
For a sequence $\vec f=\{f_j\}_{j\in\zzz}$ of tempered distributions
on $\rn$ 
we define the   smooth maximal function of  $\vec f$ with 
respect to  $\Phi$  as 
$$
{M} (\vec f\, ;\Phi)(x) = \sup_{t>0} \big\|\{(\Phi_t
*f_j)(x)\}_j\big\|_{\ell^2}
\, . 
$$
We define the  nontangential maximal function 
 with aperture $a$   of $\vec f$ with respect to  $\Phi$ as 
$$
{M}_a^*(\vec f\, ;\Phi)(x) = \sup_{t>0}
\sup_{\substack{y\in\rn\\ |y-x|\le at}} 
\big\|\{(\Phi_t *f_j)(y)\}_j\big\|_{\ell^2} \, . 
$$
We also define the auxiliary maximal function
$$
{M}_b^{**}(\vec f\, ;\Phi)(x) = \sup_{t>0} \sup_{y\in\rn}
\f{\big\|\{(\Phi_t *f_j)(x-y)\}_j\big\|_{\ell^2}}
{(1+t^{-1}|y|)^b}\, . 
$$

For a fixed positive integer $N$ we  
define the grand maximal function of $\vec f$   as
\begin{equation}\lab{6.4.4.defgrandmasxasdf}
{{\mathcal M}}_N(\vec f\, )   = \sup_{\vp  \in 
\cf_N} {{M}}_1^*(\vec f\, ;\vp)\, , 
\end{equation}
where 
\begin{equation}\lab{d6.4.1defcfN}
\cf_N =\Big\{\vp\in\cs(\rn):\,\, 
\mathfrak{N}_N (\vp)  \le 1\Big\}  \, ,
\end{equation}
and 
$$
\displaystyle\mathfrak{N}_N(\vp)=\int_{\rn}(1+|x|)^N\sum_{|\al|\le N+1}|\p^{\al}\vp(x)|\, dx .
$$
\end{definition}

If the function  $\Phi$ is not assumed
to be Schwartz but say $\Phi $ is 
the Poisson kernel, then the maximal functions  
${M}(\vec f\, ;\Phi)$,  
${M}_a^*(\vec f\, ;\Phi)$,
and 
${M}_b^{**}(\vec  f\, ;\Phi)$ 
are well defined    for    sequences of bounded tempered distributions
$\vec f=
\{f_j\}_j$. 

We note that  the following  simple inequalities 
\begin{equation}\lab{6.4.4vv}
{M}(\vec f\, ;\Phi) \le {M}_a^*(\vec f\, ;\Phi) \le
(1+a)^b {M}_b^{**}(\vec f\, ;\Phi) 
\end{equation}
are valid. 
We now define the vector-valued Hardy space $H^{p,\nf}(\rn,\ell^2)$.

\begin{definition}\lab{d6.4.12}
Let  $\vec f= \{f_j\}_j$ be a sequence  of
bounded tempered distributions 
  on $\rn$ and let $0<p<\nf$.  We say that 
$\vec f$ lies in the vector-valued weak Hardy space $H^{p,\nf}(\rn, \ell^2)$
 vector-valued Hardy space  if the 
 Poisson maximal  function 
\begin{equation*}
{M} (\vec f\, ;P)(x)= \sup_{t>0}   \big\|\{(P_t
*f_j)(x)\}_j\big\|_{\ell^2}
\end{equation*}
lies in $L^{p,\nf}(\rn)$. If this is the case, we set
$$
\big\|\vec f\, \big\|_{H^{p,\nf}(\rn, \ell^2)} =
 \big\|{ M} (\vec f\, ;P)\big\|_{L^{p,\nf}(\rn)}= 
\Big\|\sup_{\ve>0}\Big(\sum_{j}  | P_\ve*f_j|^2\Big)^{\!\f12}
\Big\|_{L^{p,\nf}(\rn)} .
$$  
\end{definition}

The next theorem provides a characterization of $H^{p,\nf}$ in terms of different maximal functions. Its proof is a copy of that for $H^p$ cases in \cite{G2}.

\begin{theorem}\lab{t6.4.13}  
Let $0<p<\nf$. Then the following statements are valid:\\
(a) There exists  a  Schwartz function $\Phi$ with integral $ 1$ 
and a  constant  $C_1 $  such that 
\begin{equation}\lab{vv}
\big\| {M}(\vec f\, ;\Phi) \big\|_{L^{p,\nf}(\rn,\ell^2)}\le  C_1
\big\|\vec f\, \big\|_{H^{p,\nf}(\rn,\ell^2)}
\end{equation}
for every sequence  $\vec f=\{f_j\}_j$ of tempered distributions. \\
(b) For every $a>0$ and
$\Phi$ in  $\cs(\rn)$   there exists a constant  $C_2( n, p,a,\Phi)$
such that 
\begin{equation}\lab{6.4.6vv}
\big\| {M}^*_a(\vec f\, ;\Phi) \big\|_{L^{p,\nf}(\rn, \ell^2)}\le 
C_2( n, p,a,\Phi) \big\| {M} (\vec f\, ;\Phi) \big\|_{L^{p,\nf}(\rn,
\ell^2)}
\end{equation}
for every sequence  $\vec f=\{f_j\}_j$ of tempered distributions. \\
(c) For every $a >0$, $b>n/p$, and
$\Phi$ in  $\cs(\rn)$   there exists a constant 
$C_3(n, p,a, b,\Phi )$ such that 
\begin{equation}\lab{6.4.7vv}
\big\| {M}^{**}_b(\vec f\, ;\Phi) \big\|_{L^{p,\nf}(\rn, \ell^2)}\le 
C_3(n, p,a, b,\Phi ) \big\|  {M}^*_a(\vec f\, ;\Phi)
\big\|_{L^{p,\nf}(\rn, \ell^2)}
\end{equation}
for every sequence  $\vec f=\{f_j\}_j$ of tempered distributions. \\
(d) For every $b >0$  and
$\Phi$ in  $\cs(\rn)$  with $\intrn\Phi(x)\, dx\neq 0$ there exists a  
 constant  $C_4(b, \Phi)$ such that if $N=[ b ]+1$ we have 
\begin{equation}\lab{6.4.8vv}
\big\|{{\mathcal{M}}}_{N}(\vec f \, ) \big\|_{L^{p,\nf}(\rn,\ell^2)}\le
C_4(b, \Phi) \big\| { M}^{**}_b(\vec f\, ;\Phi)
\big\|_{L^{p,\nf}(\rn,\ell^2)}
\end{equation}
for every sequence  $\vec f=\{f_j\}_j$ of tempered distributions.\\
(e) For every positive integer $N $ there exists  a  constant  $C_5(n,N)$
such that every sequence  
$\vec f=\{f_j\}_j$ of tempered distributions 
that satisfies $\big\| {{\mathcal{M}}}_{N}(\vec f\, )
\big\|_{L^{p,\nf}(\rn,\ell^2)}$ $<\nf$ is    
bounded and   satisfies
\begin{equation}\lab{6.4.9vv}
\big\|\vec f\, \big\|_{H^{p,\nf}(\rn,\ell^2)}  \le
C_5(n,N) \big\| {{\mathcal{M}}}_{N}(\vec f \, )
\big\|_{L^{p,\nf}(\rn,\ell^2)} \, , 
\end{equation}
that is,  it lies in the Hardy space $H^{p,\nf}(\rn,\ell^2)$. 
\end{theorem}

We conclude that for $\vec f\in H^{p,\nf}(\rn,\ell^2)$, the inequality in 
\eqref{6.4.9vv} can be reversed 
whenever $N=[\f np]+1$. 
Moreover, fix   $N=[\f np]+1$,   $\f np <b<  [\f np ]+1-\f np $, and  $\Phi$
 a Schwartz function   with $\intrn \Phi(x)\, dx =1$. Then 
for bounded distributions $\vec f=\{f_j\}$ 
 the following    equivalence of quasi-norms holds
$$
\big\| \cm_{N}(\vec f\, ) \big\|_{\wlp}\approx
\big\| M^{**}_b(\vec f;\Phi) \big\|_{\wlp} \approx
\big\| M^*_a(\vec f;\Phi) \big\|_{\wlp}\approx
 \big\| M (\vec f;\Phi) \big\|_{\wlp}   
$$
with constants that depend only on $\Phi, a, n,p$, and all the preceding quasi-norms 
are also equivalent with $  \|\vec f\, \|_{H^{p,\nf}(\rn,\ell^2)} $.

\section{Properties of $H^{p,\nf}$}
The spaces $\hi$ have several properties analogous to those of  the classical Hardy spaces $H^p$.

\begin{theorem}\label{eq} 
Let $1<p<\nf$. Then we have 
$L^{p,\iy}=\hi$ and  $\|f\|_{\wlp}\approx\|f\|_{\hi}$.
\end{theorem}
\begin{proof}
Given $f\in\wlp$, then $f$ is locally integrable, and we can define $\vp_t*f$ for a Schwartz function $\vp$ with $\int\vp\ne0$. 
By Proposition \ref{mf} which we will prove in section \ref{tool}  we have
$$\|f\|_{\hi}=\|\sup_{t>0}|(\vp_t*f)(x)|\|_{\wlp}\le C\|M(f)(x)\|_{\wlp}\le C\|f\|_{\wlp},$$
 where $M(f)$ is the Hardy-Littlewood maximal function. This shows that $f$ lies in $H^{p,\nf}$. 
 
Suppose now that $f\in \hi$. By the weak*-compactness of $\wlp=(L^{p',1})^*$, there exists a sequence $t_j\rar0$ and a function $f_0\in\wlp$ such that 
$(\vp_{t_j}*f,g)\rar(f_0,g)$ for all $g\in L^{p',1}$. This implies that $\vp_{t_j}*f\rar f_0$ in $\mcs'$.
By $\vp_t*\psi\rar\psi$ in $\mcs$ we have $\vp_t*f\rar f$ in $\mcs'$, so $f$ is in $\wlp$. 
In view of the Lebesgue differentiation theorem we obtain that 
$$\|f\|_{\wlp}\le\|\sup_{t>0}|\vp_t*f|\|_{\wlp}=\|f\|_{\hi}.$$
The preceding inequalities show that the spaces $L^{p,\nf}$ and $H^{p,\nf}$ coincide with
 equivalence of   norms.
\end{proof}

Next, we define a norm   
on Schwartz functions relevant in the theory of Hardy spaces:
$$
\mathfrak{N}_N (\vp;x_0,R)= 
 \intrn \Big(1+ \Big|\f{x-x_0}{R} \Big|\Big)^{N}
\sum_{|\al|\le N+1} R^{ |\al|}
  |\p^{\al} \vp(x)| \, dx   \, . 
$$
Note that $\mathfrak{N}_N (\vp; 0,1)=\mathfrak{N}_N (\vp )$. 

\begin{theorem}\lab{dp}
(a) For any $0<p\le 1$,  every $\vec f =\{f_j\}_j$ in $ H^{p,\nf}(\rn,\ell^2)$, and any $\vp\in\cs(\rn)$ we have 
\begin{equation}\lab{6.4.3.grafakos}
\Big(\sum_j \big|\big\langle f_j,\vp\big\rangle\big|^2\Big)^{1/2} \le 
\mathfrak{N}_N (\vp)
 \inf_{|z|\le 1} \cm_{N}(\vec f\, )(z)    \, , 
\end{equation} 
where $N=[\f np]+1$,  and consequently there is a constant $C_{n,p}$ such that 
\begin{equation}\lab{6.4.3.grafakosRRR-RRR}
\Big(\sum_j \big|\big\langle f_j,\vp\big\rangle\big|^2\Big)^{1/2} 
  \le \mathfrak{N}_N (\vp) \, C_{n,p} \, \big\| \vec f\, \big\|_{H^{p,\nf}} \, .
\end{equation} 
(b) Let $0<p\le 1 $, $N=[n/p]+1$, and $p<   r \le \nf$. Then there is a constant 
$C(p,n,r)$ such that for any $\vec f\in H^{p,\nf}$ and 
$\vp \in \mathcal S(\rn)$ we have
\begin{equation}\lab{6.4.3.grafakos-x0Rppp}
\big\|     \Big(\sum_j \big| f_j*\vp\big|^2\Big)^{1/2}      \big\|_{L^{r}} \le C(p,n,r) \mathfrak{N}_N(\vp) \big\| \vec f\,\big\|_{H^{p,\nf}}\, . 
\end{equation} 

(c) For any $x_0\in\rn$, for all $R>0$, and any $\psi \in \mathcal S(\rn)$ we have 
\begin{equation}\lab{6.4.3.grafakos-x0R}
\Big(\sum_j \big|\big\langle f_j,\psi\big\rangle\big|^2\Big)^{1/2} 
\le  \mathfrak{N}_N (\psi;x_0,R)
\inf_{|z-x_0|\le R} \cm_{N}(\vec f\, )(z) \, .  
\end{equation}

\end{theorem}

\begin{proof} 
(a) We use that $\langle f_j,\vp\big\rangle = (\wt{\vp}*f_j )(0)$, where
$\wt{\vp}(x)=\vp(-x)$ and we observe that $\mathfrak{N}_N(\vp)=\mathfrak{N}_N(\wt{\vp})$. Then 
 \eqref{6.4.3.grafakos} follows from the inequality  
$$
\Big(\sum_{j} |(\wt{\vp}*f_j )(0)|^2\Big)^{1/2} \le \mathfrak{N}_N(\vp) M^*_1\Big(\vec f; \f{\wt{\vp}}{\mathfrak{N}_N (\vp)} \Big)(z)  \le 
\mathfrak{N}_N(\vp) \mathcal M_N(\vec f\,)(z)
 $$
for all $|z-0|<1$, which is valid, since  $\wt{\vp}/\mathfrak{N}_N(\vp)$ lies in $\mathcal F_N$.  We deduce \eqref{6.4.3.grafakosRRR-RRR} as follows: set 
$\la_0=\inf_{|z|\le 1} \cm_{N}(\vec f\,)(z)$, then 
\begin{eqnarray*}
\Big(\sum_j \big|\big\langle f_j,\vp\big\rangle\big|^2\Big)^{p/2}  &\le &\mathfrak{N}_N(\vp)^p
  \inf_{|z|\le 1} \cm_{N}(f)(z)^p \\
  & \le  & \mathfrak{N}_N(\vp)^p \,  \f{1}{|B(0,1)|} |B(0,1)|\inf_{|z|\le 1} \cm_{N}(\vec f\,)(z)^p  \\
  & \le  & \mathfrak{N}_N(\vp)^p \,   \f{1}{v_n}
  \big|\{y\in \rn:\,\, \cm_{N}(\vec f\,)(y) > \la_0/2 \}\big| \la_0^p  \\
&\le  & \mathfrak{N}_N(\vp)^p\, C_{n,p}^p \, \big\|\vec f\,\big\|_{H^{p,\nf}}^p\, . 
\end{eqnarray*}

(b) For any fixed $x\in \rn$ and $t>0$ we have 
\begin{equation}\lab{6.4.pewokd}
\Big(\sum_j|(\vp_t*f_j)(x)|^2 \Big)^{1/2}\le \mathfrak{N}_N(\vp) M^*_1\Big(\vec f; \f{\vp}{\mathfrak{N}_N (\vp)} \Big)(y)  \le \mathfrak{N}_N(\vp) \mathcal M_N(\vec f\,)(y)
\end{equation}
for all $y$ satisfying $|y-x|\le 1$. Restricting to $t=1$ results  in 
$$
\Big(\sum_j|(\vp*f_j)(x)|^2\Big)^{p/2}  \le  
\mathfrak{N}_N(\vp)^p C_{p,n}^p \big\| \vec f\, \big\|_{H^{p,\nf}}^p 
$$
by an argument similar to the preceding one using $\la_0$. 
This implies that
$$ 
\bigg\| \Big(\sum_j| \vp*f_j |^2\Big)^{1/2} \bigg\|_{\li} \le C_{p,n} 
\mathfrak{N}_N(\vp)  \| \vec f\, \|_{H^{p,\nf}}.
$$
Choosing $y=x$ and $t=1$  in \eqref{6.4.pewokd} and then 
taking $\wlp$ quasinorms yields 
$$ 
\bigg\| \Big(\sum_j| \vp*f_j |^2\Big)^{1/2} \bigg\|_{\wlp} \le C_{p,n} 
\mathfrak{N}_N(\vp)  \| \vec f\, \|_{H^{p,\nf}}.
$$

By interpolation  we deduce 
$$ 
\bigg\| \Big(\sum_j| \vp*f_j |^2\Big)^{1/2} \bigg\|_{L^{r}} \le C_{p,n,r} 
\mathfrak{N}_N(\vp)  \| \vec f\, \|_{H^{p,\nf}}.
$$
when $r< p\le \nf$. 

(c) To prove \eqref{6.4.3.grafakos-x0R}, 
given a Schwartz function $\psi$ and $R>0$, define another  function $\vp$ by $\vp(y)=\psi (-Ry+x_0)$ so that 
$\psi(x) =\vp(\f{ x_0-x}{R})= R^n \vp_R ( x_0-x)$. In view of \eqref{6.4.pewokd} we have
$$
\Big(\sum_{j}\big| \big\langle f_j, \psi \big\rangle \big|\Big)^{1/2} = R^n 
\Big(\sum_{j}\big| ( \vp_R*f)(x_0) \big|\Big)^{1/2}
\le R^n\, \mathfrak{N}_N (\vp ) \inf_{|z-x_0|\le R} \mathcal M_N(\vec f\,)(z)\, . 
$$ 
But a simple change of variables shows that $R^n \,\mathfrak N(\vp)= \mathfrak N(\psi;x_0,R)$
and this combined with the preceding inequality yields \eqref{6.4.3.grafakos-x0R}. 
 \end{proof}
 
 \begin{corollary}
Convergence in $\hi$ implies convergence in $\mcs'$.
\end{corollary}
This is a direct corollary of (a) of Theorem \ref{dp}.

\begin{proposition}\label{con1}
 If $f_j\rar f$ in $\mcs'$, and $\|f_j\|_{H^p}\le C$, then $\|f\|_{H^p}\le C$. If $f_j\rar f$ in $\mcs'$, and $\|f_j\|_{\hi}\le C$, then $\|f\|_{\hi}\le C$. 
\end{proposition}
\begin{proof}
Note that $f_j\rar f$ in $\mcs'$ implies that $\vp_t*f_j\rar\vp_t*f$ pointwise for any Schwartz function $\vp$ with integral $1$. Then for any $t>0$ we have
$$
|\vp_t*f | =\liminf_{j\to \nf} |\vp_t*f_j |  \le \liminf_{j\to \nf} \sup_{s>0}|\vp_s*f_j |\, . 
$$
Taking the supremum over $t>0$ on the left and 
applying Fatou's lemma we prove this theorem.
\end{proof}

\begin{proposition} \label{norm}
The following triangle inequality holds for all $f,g$ in $H^{p,\nf}$:
$$
\|f+g\|_{\hi}^p\le2^p(\|f\|_{\hi}^p+\|g\|_{\hi}^p).
$$
Moreover,  for $0<r<p $  we have
$$
\dis\|\{f_j\}\|_{\hi}(\rn,\ell^2)\approx\sup_{0<|E|<\iy}|E|^{-\f1r+\f1p}\bigg(\int_E\sup_{t>0}\big\|\{(\vp_t
*f_j)(x)\}_j\big\|_{\ell^2}^r \, dx\bigg)^{\f1r}\, .
$$
\end{proposition}

\begin{proof}
The first claim follows from the sequence of inequalities: 
\begin{align*}
\|f+g\|_{\hi}^p=&\sup_{\la>0}\la^p|\{x:\sup_{t>0}|\vp_t*(f+g)(x)|>\la\}|\\
\le&\sup_{\la>0}\la^p|\{x:\sup_{t>0}|\vp_t*f(x)|>\tf\la2\}|+\sup_{\la>0}\la^p|\{x:\sup_{t>0}|\vp_t*g(x)|>\tf\la2\}|\\
=&2^p(\|f\|_{\hi}^p+\|g\|_{\hi}^p). 
\end{align*}
The second claim comes from the corresponding result of $\wlp$,
see \cite{G1}, p. 13.
\end{proof}

\newcommand{\wlo}{L^{1,\iy}}

\begin{proposition}
$\hi(\rn,\ell^2(L))$ are complete quasi-normed spaces.
\end{proposition}
\begin{proof} 
Consider first the case $L=1$. Let $\{f_j\}$ be a Cauchy sequence in $\hi$; then $\{f_j\}$ is also Cauchy in $\mcs'$ with   limit $f$. We use the $\|\cdot\|_{\hi}$ norm, for which we know from Proposition \ref{norm} that $\|\cdot \|_{\hi}^r$ is sublinear if $r<p$ and $r\le1$. We choose a subsequence $\{f_{j_i}\}$ of $\{f_j \}_j$ with $\|f_{j_{i+1}}-f_{j_i}\|_{\hi}^r\le 2^{-i}$, which gives us that $\|f_{j_i}\|_{\hi}^r\le C$ and therefore 
 $\|f\|_{\hi}\le C$, similarly $\|f_{j_i}-f\|\le \ep$ for any large $i$, hence $f_j\rar f$ in $H^{p,\nf}$.

Now if $\{\{f_k^{(j)}\}_{k=1}^L\}_{j=1}^{\iy}$ is Cauchy in $\hi(\rn,\ell^2(L))$, then for each $k$ we have a limit $f_k$ in $\hi$, and $\{f_k\}_{k=1}^L\in \hi(\rn,\ell^2(L))$ since $L$ is finite. If we choose $j$ large enough, we would see that 
$$
\Big\|\sup_{t>0}\Big(\sum_{k=1}^L |((f_k-f_k^{(j)})*\vp_t)(x)|^2\Big)^{\f12}\Big\|_{\wlp}^r\le \sum_{k=1}^L\|f_k-f_k^{(j)}\|^r_{\hi}\le\ep , 
$$ 
thus $\{\{f_k^{(j)}\}_{k=1}^L\}_{j=1}^{\iy}$ converge to $\{f_k\}_{k=1}^L$ in $ \hi(\rn,\ell^2(L))$ as $j\to \nf$.
\end{proof}

Next we   show that Schwartz functions are not dense in 
$\hi$ for $p\ge1$. To realize this, we investigate the decay of 
functions in $L^p$ and $\wlp$ first.

\begin{lemma}\label{LD}
For $g\in L^p(\rn)$, we have 
\begin{equation}\label{GHYTR1}
\lim_{M\rar\infty}\f{|\{|x|\le M:|g(x)|\ge |x|^{-\f n{p}}\}|}{M^n}=0,
\end{equation}
and 
\begin{equation}\label{GHYTR2}
\lim_{\de\rar0}\f{|\{|x|\le\de:|g(x)|\ge |x|^{-\f n{p}}\}|}{\de^n}=0.
\end{equation}

For $g\in \wlp$, we have 
\begin{equation}\label{GHYTR3}
\lim_{M\rar\infty}\f{|\{|x|<M:|g(x)|\ge |x|^{-\f n{p_2}}\}|}{M^n}=0,
\end{equation}
for all $p_2>p$
and 
\begin{equation}\label{GHYTR4}
\lim_{\de\rar0}\f{|\{|x|<\de:|g(x)|\ge |x|^{-\f n{p_1}}\}|}{\de^n}=0
\end{equation}
for all $0<p_1<p$.
\end{lemma}

\begin{proof}
We set $E=\{x:|g(x)|\ge |x|^{-\f n{p}}\}$.
Suppose that \eqref{GHYTR1} failed. Then there exists an $\ep>0$ and
 a sequence $\{M_k\}$ such that $\dis|E\ \cap \{x:1<|x|<M_k\}|>\ep M_k^n$ for all $M_k$; 
 moreover, we can take $M_k$ such that $\ep M_k^n>2v_nM_{k-1}^n$. Observe 
$$|E\cap\{x:M_{k-1}\le|x|\le M_k\}|>\ep M_k^n-v_nM_{k-1}^n>\f{\ep}2M_{k}^n,$$
therefore
\begin{align*}
 \intrn|g(x)|^pdx 
=\,&\,\sum_{k=2}^{\iy}\int_{\{M_{k-1}\le |x|\le M_k\}}|g(x)|^pdx\\
\ge\,&\,\sum_{k=2}^{\iy}\int_{\{M_{k-1}\le |x|\le M_k\}}|x|^{- n}\chi_Edx\\
\ge\,&\,\sum_{k=2}^{\iy}\int_{\{M_{k}-\f{\ep}2M_k\le |x|\le M_k\}}|x|^{-n}dx\\
=\,&\,C_n\sum_{k=2}^{\iy}\ln \f{1}{1-\f{\ep}2} 
= \iy
\end{align*}
This is a contradiction and our claim is true. The proof of \eqref{GHYTR2} is similar.

For an  $\wlp$ function $g$, we   cut it as follows: 
$$
g(x)=h(x)+k(x)=g(x)\chi_{\{|g|>\al\}}+g(x)\chi_{\{|g|\le\al\}},
$$
 where $\al>0$. 
Fix $p_1,\ p_2$ with $p_1<p<p_2$, then $h\in L^{p_1}$ and $k\in L^{p_2}$. Then 
\begin{align*}
\,&\,\lim_{M\rar\iy}\f{|\{|x|<M:g>|x|^{-\f n{p_2}}\}|}{M^n}\\
\le\,&\,\lim_{M\rar\iy}\f{|\{1<|x|<M:2|h|>|x|^{-\f n{p_2}}\}|}{M^n}+\lim_{M\rar\iy}\f{|\{1<|x|<M:2|k|>|x|^{-\f n{p_2}}\}|}{M^n}\\
\le \,&\,\lim_{M\rar\iy}\f{|\{1<x<M:2|h|>x^{-\f n{p_1}}\}|}{M^n}+\lim_{M\rar\iy}\f{|\{1<x<M:2|k|>|x|^{-\f n{p_2}}\}|}{M^n}\\
=\,&\,0.
\end{align*}
This proves \eqref{GHYTR3} and \eqref{GHYTR4} can be proved in a similar way.

\end{proof}

\begin{theorem}\label{RP}
  $L^r$ is not dense in $\wlp$, whenever $0<r\le\iy$ and $0<p<\iy$. 
\end{theorem}

\begin{proof}
For simplicity we  restrict ourselves to the case $p=1$ and $n=1$, and we note that the proof in this case contains the general idea. 

For the case $r\ge1=p$, we don't need the decay we proved in Lemma \ref{LD}.
We will prove a stronger result: $L^1_{\loc}$ is not dense in $L^{1,\iy}$. Suppose not, then in $\wlo$ the function $ f(x)=\tf{1}{x}\chi_{(0,1)}\in L^{1,\iy}$ is
in the closure of $L^1_{\loc}$. The fact that $\|f-g\|_{\wlo}\ge\|f-g\chi_{[0,1]}\|_{\wlo}$ for all $g\in L^1_{\loc}$ 
suggests us that $f$ is also in the closure of $L^1_{\loc}[0,1]=L^1[0,1]$. Moreover step functions are dense in $L^1$ and $\|h\|_{\wlo}\le\|h\|_{L^1}$, 
so $f$ is also in the closure of set of step functions defined on $[0,1]$. But such a step function $g$ must be bounded by $M$, then 
$$
\|f-g\|_{\wlo}=\|(f-M)\chi_{(0,\tf1{2M})}\|_{\wlo}\ge\f12.
$$
This contradiction shows that $L^1_{\loc}$ is not dense in $\wlo$. In particular, $L^r$ is not dense in $\wlo$ if $r\in[1,\iy]$.

If $1=p>r$, then we can take $f=\f{1}{x}\chi_{(1,\iy)}\in \wlo$ with $\|f\|_{\wlo}=1$. 
Choose $M$ large enough so that $x^{-1}-x^{-\tf1r}>\f12x^{-1}$ for $x\ge M$. Fix any $g\in L^r$ and by Lemma \ref{LD} 
we can choose $M'>10 M$ such that
$$
\f{|\{1<x<M':|g(x)|>x^{-\f1r}\}|}{M'}\le\f12,
$$
therefore we can estimate the difference of $f$ and $g$ as  follows:
\begin{align*}
\|f-g\|_{\wlp}\ge&\|(|f|-|g|)\chi_{(M,M')}\|_{\wlp}\\
=&\sup_{\al>0}|\{M'>x>M:|g(x)|\le x^{-\f1r},\ x^{-1}-x^{-\f1r}>\al\}|^{\f1p}\al\\
\ge&\sup_{\al>0}|\{M'>x>M:|g(x)|\le x^{-\f1r},\ \f12x^{-1}>\al\}|^{\f1p}\al\\
\ge&(\tf25M')\al_0 
= \tf1{10}
\end{align*}
where we took $\f{{M'}^{-1}}{2}>\al_0>\f{{M'}^{-1}}{4}$.
\end{proof}

We want to take the more general result we proved as a corollary.
\begin{corollary}
$L^p_{loc}\cap \wlp$ is not dense in $\wlp$.
\end{corollary}
We also want to remark that to prove this corollary we can also use the decay we proved in Lemma \ref{LD}. More concretely we can consider the distance of 
$f(x)=x^{-1/p}\chi_{(0,1)}$ and any $L^p_{\loc}$ function.

We have showed   that Schwartz functions are not dense in $H^{p,\infty}$ for $p>1$ since $\wlp=\hi$. 
Unlike the situation for strong Hardy spaces where $H^1\subset L^1$, the distribution $ \delta_1-\delta_{-1}$ is in $H^{1,\iy}$ but not in $L^{1,\iy}$. This 
shows that $H^{1,\iy}$ is not a subspace of $L^{1,\iy}$. 
Therefore we need a different method to show that Schwartz functions are not dense in $H^{1,\iy}$. Actually   
the preceding distribution   can be used to 
show this.

\begin{theorem}\label{SD}
The space of Schwartz functions $\mcs$ is not dense in $H^{1,\iy}$. 
Moreover $L^1$ is not dense in $H^{1,\nf}$.
\end{theorem}
\begin{proof}
We will provide a constructive proof. More concretely, 
we will show that for any $\vp\in\mcs$, $\|f-\vp\|_{H^{1,\nf}}\ge
\f1{20}$, 
where $f=\de_1-\de_{-1}$. An easy calculation shows that 
$$
f^+(x)=\sup_{t>0}|(\psi_t*f)(x)|\approx\left\{
\begin{array}{ll}
\f{|x|}{|x-1|^2|x+1|}&x<0,\\
\\
\f{|x|}{|x+1|^2|x-1|}&x>0.
\end{array}\right.
$$
By the symmetry of this function, we can consider only the part $x>0$. So $f^+(x)=\f{x}{(x+1)^2(1-x)}$ for $0<x<1$ and $\f{x}{(x+1)^2(x-1)}$ for $x>1$. 
If we cut this function in  two parts, $f^+(x)=g(x)+h(x)$ with $g(x)=f^+(x)\chi_{|x-1|<a}+f^+(x)\chi_{x>b}$ for fixed small $a$ and large $b$, say $\f12$ and $100$,
 then $h$ is bounded and compactly supported and lies in any $L^{p,\iy}$. 
 
 Now we consider $g(x)$. If $\al$ is large, then 
 $$
 |\{x:g(x)>\al\}|\le |\{x>0:\,\, \tf89\tf1{|1-x|}>\al\}|\le\tf{16}{9\al},
 $$
 therefore $ \sup_{\al>\f1{1000}}\al^p\f{16}{9\al}<\iy$ for all $p\le1$. Meanwhile, if $\al$ is small, then 
$$
|\{x:g(x)>\al\}|\le 1+|\{x>100:\,\, \tf2{x^2}>\al\}|\le\sqrt{\tf{2}{\al}}-99,
$$
and we have $ \sup_{\al\le\f{1}{1000}}\al^p(\sqrt{ {2}/{\al}}-99)<\iy$ for $p\ge\f12$. 
In conclusion, $g\in L^{p,\iy}$ for $p\in[\f12,1]$ and $f\in \hi$ for the same range of $p$'s.

Next, we   show that 
$\|f-\vp\|_{H^{1,\nf}}>\f1{20}$. We achieve this via the estimate
$$
\|g-\vp^+(\chi_{|x-1|<a}+\chi_{x>b})\|_{\wlp}\le\|f^+-\vp^+\|_{\wlp}\le\|f-\vp\|_{\hi}.
$$

For $1\le x\le\f32$, we have $\f{x}{(x+1)^2(1-x)}>\f{4}{25}\f{1}{(x-1)}$. So $\|\vp^+\|_{L^{\iy}}\le C$, and thus  
$$
 \|g-\vp^+(\chi_{|x-1|<a}+\chi_{x>b})\|_{L^{1,\iy}} 
\ge \sup_{\al>100}\al|\{x\in(1,\tf32):\tf{4}{25(x-1)}-C>\al\}|>\f2{25}.
$$
This proves the required claim 
for Schwartz functions.  

For functions in  $L^1$ we argue as follows.
Keep $f=\de_1-\de_{-1}$ as before and let $g$ be any $L^1$ function. Actually $g$ is also in $H^{1,\nf}$
since
$$\|g\|_{H^{1,\nf}}=\|M(g;P)\|_{L^{1,\nf}}
\le \|M(g)\|_{L^{1,\nf}}\le C\|g\|_{L^1},$$
where $M$ is the Hardy-Littlewood maximal operator.
Now
for any given $\ep>0$,
 we can choose 
a Schwartz function $\vp$ such that $\|\vp-g\|_{L^1}\le \ep$.
Apply the previous discussion we deduce that
$$\|\vp-g\|_{H^{1,\nf}}
\le C\ep.$$
By Proposition \ref{norm}, 
$\|f-\vp\|_{H^{1,\nf}}\le 2(\|f-g\|_{H^{1,\nf}}+\|g-\vp\|_{H^{1,\nf}})$, 
but $\|f-\vp\|_{H^{1,\nf}}\ge
\f1{20}$, therefore
$$\|f-g\|_{H^{1,\nf}}\ge\tf12\|f-\vp\|_{H^{1,\nf}}-\|g-\vp\|_{H^{1,\nf}}\ge 1/40-C\ep\ge 1/80$$
if we choose $\ep$ small enough.
\end{proof}

\section{Two interpolation results}\label{tool}\label{Int}

The following version of the classical Fefferman-Stein vector-valued inequality \cite{FSvv} will be 
necessary in our work. 
This result for upper Boyd indices $r$ of $L^{p,r}$, which is less than $\nf$, is contained in \cite{CUMP} (page 85) 
and here we provide the proof for the case $r=\nf$  not contained in \cite{CUMP}.

\begin{proposition}\label{mf}
 If $1<p,q<\nf $, then for all sequences of functions $\{f_j\}_j$ in $L^{p,\iy}(\ell^q)$ we have 
$$
\left\|\|\{M(f_j)\}\|_{\ell^q}\right\|_{L^{p,\iy}}\le C_{p,q}\left\|\|\{f_j\}\|_{\ell^q}\right\|_{L^{p,\iy}},
$$
where $M$ is the Hardy-Littlewood maximal function.
\end{proposition}

\begin{proof}
We   know  that $\left\|\|\{M(f_j)\}\|_{\ell^q}\right\|_{\lp}\le C_{p,q}\left\|\|\{f_j\}\|_{\ell^q}\right\|_{\lp}$ for $1<p,q<\nf $, see \cite{FSvv}. Now fix $q$ and take $1<p_1<p<p_2$. Set 
$\vec F=\{f_j\}$    and $|\vec F|= (\sum_{j\in\zzz}|f_j|^q)^{\f1q}$. We   split $\vec F$ at the height $\al>0$, and define $\vec F_{\al}=\vec F\chi_{|\vec F|>\al}$ 
and $\vec F^{\al}=\vec F-\vec F_{\al}=\vec F\chi_{|\vec F|\le \al}$.

It is easy to verify that 
$$
|\{|\vec F_{\al}|>\la\}|=
\left\{
\begin{array}{ll}
d_{\vec F}(\la)&\la>\al\\
d_{\vec F}(\al)&\la\le\al
\end{array}
\right.
$$
and
$$
|\{|\vec F^{\al}|>\la\}|=
\left\{
\begin{array}{ll}
0&\la>\al\\
d_{\vec F}(\la)-d_{\vec F}(\al)&\la\le\al
\end{array}
\right.,
$$
where $d_{\vec F}(\la)=|\{|\vec F|>\la\}|$. Consequently, 
$
\left\||\vec F_{\al}|\right\|_{L^{p_1}}^{p_1}
\le \f{p}{p-p_1}\al^{p_1-p}\left\||\vec F|\right\|_{\wlp}^p
$
and 
$
\left\||\vec F^{\al}|\right\|_{L^{p_2}}^{p_2}
\le \f{p_2}{p_2-p}\al^{p_2-p}\left\||\vec F|\right\|_{\wlp}^p-d_{\vec F}(\al)\al^{p_2}.
$
For each $j$, split $f_j=f_j\chi_{|\vec F|>\al}+f_j\chi_{|\vec F|\le\al}$. Then we have 
\begin{align*}
&\left|\{\|\{M(f_j)\}\|_{\ell^q}>\la\}\right|\\
\le&|\{\|\{M(f_j)\chi_{|\vec F|>\al}\}\|_{\ell^q}>\tf{\la}{2}\}|+|\{\|\{M(f_j)\chi_{|\vec F|\le\al}\}\|_{\ell^q}>\tf{\la}{2}\}|\\
\le&C(p_1,q)(\tf2\la)^{p_1}\int_{\rn}\|\{f_j\chi_{|\vec F|>\al}(x)\}\|^{p_1}_{l^q}dx+C(p_2,q)(\tf2\la)^{p_2}\int_{\rn}\|\{f_j\chi_{|\vec F|\le\al}(x)\}\|^{p_2}_{\ell^q}dx\\
\le&C(p_1,q)(\tf2\la)^{p_1}\tf{p}{p-p_1}\al^{p_1-p}\left\||\vec F|\right\|_{\wlp}^p+C(p_2,q)(\tf2\la)^{p_2}\tf{p_2}{p_2-p}\al^{p_2-p}\left\||\vec F|\right\|_{\wlp}^p\\
\le& \big(\f{p}{p-p_1}+\f{p_2}{p_2-p}\big)2^{p_2}C(p_1,q)^{\tf{p_2-p}{p_2-p_1}}C(p_2,q)^{\tf{p-p_1}{p_2-p_1}}\la^{-p}\left\|\vec F\right\|_{\wlp}^p
\end{align*}
where we set $\al=\la\ga$, where $\ga=(\tf{C(p_1,q)}{C(p_2,q)})^{\tf1{p_2-p_1}}$.
\end{proof}

\smallskip


Next we have the following result which has a lot of applications in this work. The scalar version of this theorem
has been proved in \cite{FRS}, but it is incomplete in the case $q=\iy$ due to the fact that Schwartz functions are not  dense in $\hi$; this is shown in 
Theorem \ref{RP} and Theorem \ref{SD}. Here we complete this gap  pushing further the approach  \cite{FRS} and
combining it with ideas from  chapter III of \cite{S1}.

\bt\label{In1}
(a) Let $J$ and $L$ be positive integers and let $0<p_1<p<p_2<\nf$, moreover $p_1\le1$.
  Let $T$ be a sublinear operator 
defined on $H^{p_1}(\rn,\ell^2(L))+ H^{p_2}(\rn,\ell^2(L))$. Assume 
  that maps $H^{p_1}(\rn,\ell^2(L))$ to $H^{p_1}(\rn,\ell^2(J))$ with constant $A_1$ and $H^{p_2}(\rn,\ell^2(L))$ to $H^{p_2}(\rn,\ell^2(J))$ with constant $A_2$. 
Then there exists a constant $c_{p_1,p_2,p,n}$ independent of $J$ and $L$ such that 
$$
\|T(\vec F)\|_{\hi(\rn,\ell^2(J))}\le c_{p_1,p_2,p,n}\, A_1^{ \f{\f 1{p}-\f 1 {p_2}}{\f 1{p_1}-\f 1{p_2}}}A_2^{ \f{\f 1{p_1}-\f 1 p}{\f 1{p_1}-\f 1{p_2}}}\|\vec F\|_{\hi(\rn,\ell^2(L))}
$$
 for $\vec F\in\hi(\rn,\ell^2(L))$. 
 \et
(b) Suppose that  $T$ is a sublinear operator  
defined on $H^{p_1}(\rn,\ell^2(L))+ H^{p_2}(\rn,\ell^2(L))$. Assume
 that maps $H^{p_1}(\rn,\ell^2(L))$ to $L^{p_1}(\rn,\ell^2(J))$ with constant $A_1$ and $H^{p_2}(\rn,\ell^2(L))$ to $L^{p_2}(\rn,\ell^2(J))$ with constant $A_2$. 
Then there exists a constant $C$ independent of $J$ and $L$ such that 
$$
\|T(\vec F)\|_{\wlp(\rn,\ell^2(J))}\le c_{p_1,p_2,p,n}\, A_1^{ \f{\f 1{p}-\f 1 {p_2}}{\f 1{p_1}-\f 1{p_2}}}A_2^{ \f{\f 1{p_1}-\f 1 p}{\f 1{p_1}-\f 1{p_2}}}\|\vec F\|_{\hi(\rn,\ell^2(L))}
$$
 for all functions $\vec F\in\hi(\rn,\ell^2(L))$.  

\begin{lemma}\label{cz}
Let $0<p_1<p<p_2<\nf$. 
Given $\vec F=\{f_k\}_{k=1}^L\in\hi(\rn,\ell^2(L))$ and $\al>0$, then there exists $\vec G=\{g^{k}\}_{k=1}^L$ and $\vec B=\{b^{k}\}_{k=1}^L$ such that $\vec F=\vec G+\vec B  $ and 
$$
\|\vec B\, \|_{H^{p_1}(\rn,\ell^2(L)) }^{p_1}\le C\al^{p_1-p}\|F\|_{\hi(\rn,\ell^2(L))}^p
$$
and  
$$
\| \vec G \|_{H^{p_2}(\rn,\ell^2(L))}^{p_2}\le C\al^{ {p_2-p} }
\|\vec F\|_{H^{p,\iy}(\rn,\ell^2(L))}^{ {p} }
$$
where $C=C(p_1,p_2,p,n)$, in particular is independent of $L$.
\end{lemma}

\begin{proof}[Proof of theorem \ref{In1}]
Suppose that $T=\{T_j\}_{j=1}^J $. 
We apply Lemma \ref{cz} with $\al = \ga \la$ where $\ga=(A_2^{p_2} A_1^{-p_1})^{\f 1{p_1-p_2}}$.  We obtain 
\begin{align*}
&|\{x:\sup_{t>0}(\sum_{j=1}^J|T_j(\vec F)*\psi_t(x)|^2)^{\f12}>\la\}|\\
\le&|\{x:\sup_{t>0}(\sum_{j=1}^J|T_j(\vec G)*\psi_t(x)|^2)^{\f12}+\sup_{t>0}(\sum_{j=1}^J|T_j(\vec B)*\psi_t(x)|^2)^{\f12}>\la\}|\\
\le&|\{x:\sup_{t>0}(\sum_{j=1}^J|T_j(\vec G)*\psi_t(x)|^2)^{\f12}>\la/2\}|+|\{x:\sup_{t>0}(\sum_{j=1}^J|T_j(\vec B)*\psi_t(x)|^2)^{\f12}>\la/2\}|\\
\le &(\tf{2}{\la})^{p_2}\|\sup_{t>0}(\sum_{j=1}^J|T_j(\vec G)*\psi_t(x)|^2)^{\f12}\|_{L^{p_2}}^{p_2}+(\tf{2}{\la})^{p_1}\|\sup_{t>0}(\sum_{j=1}^J|T_j(\vec B)*\psi_t(x)|^2)^{\f12}\|_{L^{p_1}}^{p_1}\\
\le&A_2^{p_2}(\tf{2}{\la})^{p_2}\|\vec G\|_{H^{p_2}(\rn,\ell^2(L))}^{p_2}+A_1^{p_1}(\tf{2}{\la})^{p_1}\|\vec B\|_{H^{p_1}(\rn,\ell^2(L))}^{p_1}\\
\le&  A_2^{p_2}(\tf{2}{\la})^{p_2}\|\vec F\|_{\hi(\rn,\ell^2)}^pC\,(\ga\la)^{p_2-p}+A_1^{p_1}(\tf{2}{\la})^{p_1}\|\vec F\|_{\hi(\rn,\ell^2)}^pC\,(\ga\la)^{p_1-p}     \\
\le & C  2^{p_2+1} (A_1^{1-\theta} A_2^{\theta})^p\la^{-p}\|\vec F\|_{\hi}^p  \, , 
\end{align*}
where $\theta = \f{\f 1{p_1}-\f 1 p}{\f 1{p_1}-\f 1{p_2}}$.

The proof of the second part is similar,
\begin{align*}
&|\{x:(\sum_{j=1}^J|T_j(\vec F)(x)|^2)^{\f12}>\la\}|\\
\le&|\{x:(\sum_{j=1}^J|T_j(\vec G)(x)|^2)^{\f12}+(\sum_{j=1}^J|T_j(\vec B)(x)|^2)^{\f12}>\la\}|\\
\le&|\{x:(\sum_{j=1}^J|T_j(\vec G)(x)|^2)^{\f12}>\la/2\}|+|\{x:(\sum_{j=1}^J|T_j(\vec B)(x)|^2)^{\f12}>\la/2\}|\\
\le &(\tf{2}{\la})^{p_2}\|(\sum_{j=1}^J|T_j(\vec G)(x)|^2)^{\f12}\|_{L^{p_2}}^{p_2}+(\tf{2}{\la})^{p_1}\|(\sum_{j=1}^J|T_j(\vec B)(x)|^2)^{\f12}\|_{L^{p_1}}^{p_1}\\
\le&A_2^{p_2}(\tf{2}{\la})^{p_2}\|\vec G\|_{H^{p_2}(\rn,\ell^2(L))}^{p_2}+A_1^{p_1}(\tf{2}{\la})^{p_1}\|\vec B\|_{H^{p_1}(\rn,\ell^2(L))}^{p_1}\\
\le&  A_2^{p_2}(\tf{2}{\la})^{p_2}\|\vec F\|_{\hi(\rn,\ell^2)}^pC\,(\ga\la)^{p_2-p}+A_1^{p_1}(\tf{2}{\la})^{p_1}\|\vec F\|_{\hi(\rn,\ell^2)}^pC\,(\ga\la)^{p_1-p}     \\
\le & C  2^{p_2+1} (A_1^{1-\theta} A_2^{\theta})^p\la^{-p}\|\vec F\|_{\hi}^p  \, , 
\end{align*}

\end{proof}

\begin{proof}[Proof of lemma \ref{cz}]
We introduce the  notation 
$$
{\vec F}^* (x)=\mathcal{M}_N(\vec F)(x)=\sup_{\mathcal{N}_N(\phi)\le1}\sup_{|x-y|\le t}\Big(\sum_{k=1}^L|(\phi_t*f_k)(y)|^2\Big)^{\f12}
$$ 
for the grand maximal function and 
$$
\mathcal{M}_0(\vec F)(x)
=\sup_{t>0}\Big(\sum_{k=1}^L|(\psi_t*f_k)(x)|^2\Big)^{\f12} 
$$
for a maximal function with respect to a fixed bump $\psi$. 
It's easy to check that $ \Om_{\al}=\{{\vec F}^*(x)>\al\}$ is open, so we can use the Whitney 
decomposition
theorem to get a collection of cubes $Q_j$ and functions $\vp_j$ such that
\begin{enumerate}
\item $\bigcup_jQ_j=\Om_{\al}$,
\item the $Q_j$ are mutually disjoint,
\item diam($Q_j$)$<$ dist($Q_j,\Om_{\al}^c$)$\le$ $4$ diam($Q_j$), 
 where diam$(Q_j)$ is the diameter of $Q_j$ which is $\sqrt nl_j$
with $l_j$ the length of $Q_j$,
\item every point is contained in at most $12^n$ cubes of the form $Q_j^*=a\, Q_j$ with $a-1>0$ is fixed and small, 
\item $|(\f{\p}{\p x})^{\be}\vph_j(x)|\le A_{\be}l_j^{-|\be|}$, 
where $A_{\be}$ is a constant independent of $l_j$,
\item $supp\ \vph_j\subset Q_j^{**}=b\, Q_j$ ($1<b<a$) and  there exists $0<c<1$ depending on $n$ such that for all
$j$, $\vph_j\ge c>0$ for $x\in Q_j$.
\item $\sum_j\vph_j=\chi_{\Om_{\al}}$.
\end{enumerate}

Here $a\, Q$ is a cube concentric with $Q$ and of side length $a $ times that of $Q$. 
Next we will define $b_j^{k}$ and show the corresponding estimates. Fix $j$, define $P^{(k)}_j$ as the polynomial of degree $N$ , where $N$ is a 
fixed large integer to be chosen, such that
$$
\intrn P^{(k)}_j(x)(x-x_j)^{\be}\vp_j(x)\, dx= \langle f_k,(x-x_j)^{\be}\vp_j\rangle  \q\forall\ |\be|\le N,
$$
where $x_j$ is the center of $Q_j$, and $<f,\vp>$ is the action of $f$ on $\vp$.

Take the norm of $h$ in the Hilbert space of polynomials of degree $\le N$ as 
$$\left\|h\right\|^2=\f{\intrn |h|^2(x)\vp_j(x)dx}{\intrn\vp_j(x)dx}.$$
We have an orthonormal basis $\{e_m\}$ of this Hilbert space with each $e_m$ is a polynomial of degree less than or equal to $  N$ and $\|e_m\|=1$.
Also we can write $P_{j}^{(k)}(x)=\sum_m(f_k,e_m)_je_m(x)$, 
where $(f,h)_j$ is the inner product defined by $(f,h)_j= \f{\langle f,h\vp_j \rangle}{\int \vp_j}$. It's not hard to check 
 the following inequality by the method from \cite{S1},
$$
\sup_{x\in Q^*_j}|\p^{\be}P(x)|\le C_{\be}l_j^{-|\be|}\left(\f{\int_{Q_j}|P|^2}{|Q_j|}\right)^{\f12} \textrm{ for any $P$ with degree $\le N$}.
$$
To prove it, we can reduce this to the case $Q_j$ is a unit cube and prove this case by
that different norms of a finite dimension topological vector space are comparable.
Let's notice also that $\Big(\f{\int_{Q_j}|P|^2}{|Q_j|}\Big)^{\f12}\le C\|P\|$. 
In particular, $|\bar e_m(x)|\le C,\ \forall\ x\in Q_j^*$ and 
$|\p^{\ga}e_m(y)|\le C l_j^{-|\ga|},\ \forall\ y\in Q_j^*$ .
Now let's restrict $x\in Q_j^*$  and we have
\begin{align*}
\Big(\sum_{k=1}^L|P_j^{(k)}(x)|^2\Big)^{\f12}=&\Big(\sum_{k=1}^L|(f_k,\sum_me_m(\cdot)\bar e_m(x))_j|^2\Big)^{\f12}\\
=&\left(\sum_{k=1}^L\left|\left(\f{\langle f_k,\sum_me_m(\cdot)\bar e_m(x)\vp_j(\cdot) \rangle}{\int\vp_j}\right)\right|^2\right)^{\f12}\\
\le& C_N\inf_{|z-x|\le 10\sqrt nl_j}\vec F^*(z)
\le   C_N\al \, .
\end{align*}
Here we introduced the function 
$
\Phi(x,y)=\f{\sum_me_m(y)\bar e_m(x)\vp_j(y)}{\int\vp_j}
$
for which the estimate below holds
$$
\mathfrak{N}_N(\Phi;x,10\sqrt nl_j)
\le C_N\int_{Q_j^*}\sum_{|\ga|\le N+1}l_j^{|\ga|}l_j^{-n}l_j^{-|\ga|}dy
\le C_N\, . 
$$
Let's remark that this $C_N$ is independent of $L$.
Now let's define $b_j^{(k)}=(f_k-P_j^{(k)})\vp_j$ and consider $ (\sum_{k=1}^L|(b_j^{(k)}*\psi_t)(x)|^2)^{\f12}$, where $\psi$ is a smooth function supported in 
$B(0,1)\backslash B(0,\f12)$ with $\int \psi\neq0$.
If $x\in Q_j^*$, then 
$$\Big(\sum_{k=1}^L|(b_j^{(k)}*\psi_t)(x)|^2\Big)^{\f12}\le  \Big(\sum_{k=1}^L|((P_j^{(k)}\vp_j)*\psi_t)(x)|^2\Big)^{\f12}+
\Big(\sum_{k=1}^L|\langle f_k,\vp_j(\cdot)\psi_t(x-\cdot) \rangle |^2\Big)^{\f12}.$$
We know that the first term of the right hand side is controlled by $C_N \al$ since so is $(\sum_{k=1}^L|P_j^{(k)}(x)|^2)^{\tf12}$ for all $x\in Q_j^*$.
 For the second term, if $t\le l_j$, then take $\Phi(y)=\vp_j(x-ty)\psi(y)$, 
 and by Theorem \ref{dp} it is easy to
check the following inequalities
\begin{equation*}
 \Big(\sum_{k=1}^L| \langle f_k,\vp_j(\cdot)\psi_t(x-\cdot) \rangle |^2\Big)^{\f12} 
\le \mathfrak{N}_N(\Phi)\vec F^*(x)
\le  C\vec F^*(x)
\end{equation*}
If $t>l_j$, we can use the same idea but $\Phi(y)=\vp_j(x-l_jy)\psi(\f{l_j}{t}y)$ to get that 
$$
\Big(\sum_{k=1}^L| \langle f_k,\vp_j(\cdot)\psi_t(x-\cdot) \rangle |^2\Big)^{\f12} \le C\vec F^*(x).
$$

If $x\in (Q_j^*)^c$, then there exists $C$ such that supp$\psi_t\cap Q_j^{**}=\eset$
if $t\le Cl_j$, from which we have $(b_j^{(k)}*\psi_t)(x)=0$
since $b_j^{(k)}$ is a distribution supported in $Q_j^{**}$.  
So we can assume $t\ge Cl_j$ in the following discussion.
Now let us fix $x$ and  write $t^{-n}\psi(\f{x-y}{t})=P(y)+R(y)$ by Taylor's formula, where 
$$
P(y)=\sum_{|\be|\le N}\f{\p^{\be}h(x_j)(y-x_j)^{\be}}{\be!}
$$ 
with $h(y)=t^{-n}\psi((x-y)/t)$
is the Taylor polynomial
of degree $N$ at $x_j$ and $R$ is the remainder. Next we concern only  
$y$ in the support of $\vp_j$ because, as we will see, $y$ such that $y\notin Q_j^{**}$ does not affect the following argument.
It is easy to see that
 \begin{align*}
 \Big(\sum_{k=1}^L|(b_j^{(k)}*\psi_t)(x)|^2\Big)^{\f12}=
& \left(\sum_{k=1}^{L}|<(f_k-P_j^{(k)})\vp_j,\psi_{t}(x-\cdot)>|^2\right)^{\f12}\\
=&\left(\sum_{k=1}^{L}|<(f_k-P_j^{(k)}),\vp_jR>|^2\right)^{\f12}\\
\le&\left(\sum_{k=1}^{L}|<f_k, \vp_jR>|^2\right)^{\f12}+\left(\sum_{k=1}^{L}|<P_j^{(k)},\vp_jR>|^2\right)^{\f12}.
 \end{align*}
Since $(\sum_{k=1}^L|P_j^{(k)}(x)|^2)^{\tf12}\le C\al$ for $x\in Q^*_j$,
$<P_j^{(k)},\vp_jR>$ can be written as an integral and 
$
(\sum_{k=1}^{L}|<P_j^{(k)},\vp_jR>|^2)^{\f12}\le t^{-n}C\al l_j^n\le C\al(l_j/(|x-x_j|))^{N+n+1}.
$

We also have the estimate for $y\in Q_j^{**}$ 
\begin{equation}\label{rmd}
|\p^{\ga}R(y)|\le Cl_j^{N+1-|\ga|}/(|x-x_j|)^{N+n+1}.
\end{equation}
Indeed, for $|\ga|\le N$, if we apply Taylor's formula to $\p^{\ga}h$ again, we will have 
$$
\p^{\ga}h(y)=\p^{\ga}P(y)+\sum_
{|\be|=N-|\ga|+1}\f{\p^{\be+\ga}h(\xi)(y-x_j)^{\be}}{\be!},
$$
where $\xi$ is a point between $x_j$ and $y$. In other words,
$$
\p^{\ga}R(y)=\p^{\ga}(h-P)(y)=\sum_
{|\be|=N-|\ga|+1}\f{\p^{\be+\ga}h(\xi)(y-x_j)^{\be}}{\be!}.
$$
Notice 
that $|y-x_j|\le Cl_j$ and 
$|x-x_j|\le Ct$,  then for $|\be|=N-|\ga|+1$,
$$
|\p^{\be+\ga}h(\xi)(y-x_j)^{\be}|\le t^{-n}t^{-N-1}|y-x_j|^{N-|\ga|+1}\le Cl_j^{N+1-|\ga|}|x-x_j|^{-N-n-1}.
$$
Consequently $|\p^{\ga}R(y)|\le Cl_j^{N+1-|\ga|}|x-x_j|^{-N-n-1}$ if $|\ga|\le N$.


Since $|x-x_j|\le|x-y|+|x_j-y|\le |x-y|+C_1l_j\le C'|x-y|$,
for the case $N+1-|\ga|\le 0$, we have
$$|\p^{\ga}R(y)|=|\p^{\ga}h(y)|\le t^{-n}t^{-|\ga|}(\f{t}{|x-y|})^{N+n+1}\le Cl_j^{N+1-|\ga|}|x-x_j|^{-N-n-1}.$$

Now let us take $\Phi(z)=R(z)\vp_j(z)$, and by Theorem \ref{dp} (c) we obtain
\begin{align*}
&\mathfrak{N}_N(\Phi; x_j,10\sqrt nl_j)\\
=&\intrn(1+\f{|z-x_j|}{10\sqrt nl_j})^N\sum_{|\ga|\le N+1}(10\sqrt nl_j)^{|\ga|}|\p^{\ga}\Phi(z)|dz\\
\le&\int_{|z-x_j|\le Cl_j}(1+\f{|z-x_j|}{10\sqrt nl_j})^N\sum_{|\ga|\le N+1}(10\sqrt nl_j)^{|\ga|}|\p^{\ga}\Phi(z)|dz\\
\le &C(\f{l_j}{|x-x_j|})^{N+n+1}
\end{align*} 
since 
$(10\sqrt nl_j)^{|\ga|}|\p^{\be}R(z)\p^{\ga-\be}\vp_j|\le C l_j^{|\ga|}l_j^{N+1-|\be|}|x-x_j|^{-N-n-1}l_j^{|\be-\ga|}$
by \eqref{rmd},
and that $|\{z\in\rn:|z-x_j|\le Cl_j\}|=Cl_j^n$.

If we take $y_j$ as a point in $\Om_\al^c$ with $|y_j-x_j|\le 10\sqrt nl_j$ 
and apply the idea used in two previous cases again, we have 
 \begin{align*}
 \left(\sum_{k=1}^{L}|<f_k, \vp_jR>|^2\right)^{\f12}
 \le &C(\f{l_j}{|x-x_j|})^{N+n+1}\inf_{|w-x_j|\le 10\sqrt nl_j} \vec F^*(w) \, \\
 \le& C(\f{l_j}{|x-x_j|})^{N+n+1}{\vec F}^* (y_j)\\
 \le& C\al (\f{l_j}{|x-x_j|})^{N+n+1}.
 \end{align*}
 
 To summarize, $ (\sum_{k=1}^L|(b_j^{(k)}*\psi_t)(x)|^2)^{\f12}\le C\al (\f{l_j}{|x-x_j|})^{N+n+1}$ for $x\in (Q_j^*)^c$.

 If $N$ is chosen such that $(N+n+1)p_1>n$, then by $p_1\le 1$ we will get 
\begin{align*}
&\intrn\sup_{t>0}(\sum_{k=1}^{L}|(\sum_{m_1\le j\le m_2}b_{j}^{(k)}*\psi_t)(x)|^2)^\f{p_1}2dx\\
\le &\sum_{j=m_1}^{m_2}\left(\int_{Q_j^{*}}\vec F^*(x)^{p_1}dx+C\al^{p_1}\int_{(Q_j^*)^c}\Big(\f{l_j}{|x-x_j|}\Big)^{(N+n+1)p_1}dx\right)\\
\le& C\sum_{j=m_1}^{m_2}\int_{Q_j^{*}}\vec F^*(x)^{p_1}dx\\
\le& C\int_{\Om_\al}\vec F^*(x)^{p_1}dx\\
\le&C \al^{p_1-p}|\Om_\al|^{1-\f{p}{p_1}}\left(\int_{\Om_\al}\vec F^*(x)^{p_1}dx\right)^{\f{p}{p_1}}\\
\le &C\al^{p_1-p}\|\vec F^*\|_{\wlp}^p<\iy\, .
\end{align*}

The penultimate inequality comes from an equivalent definition of $\wlp$ spaces, see \cite{G1} p. 13.
When $m_1=1$, $\int\sum_{j=1}^{m_2}\chi_{Q^*_j}\vec F^*(x)^{p_1}dx\le \int C\chi_{\Om_{\al}}\vec F^*(x)^{p_1}dx<\nf$
by the decomposition of $\Om_{\al}$,
then apply the Lebesgue dominate convergence theorem, $\int\sum_{j=1}^{\nf}\chi_{Q^*_j}\vec F^*(x)^{p_1}dx<\nf $ and
$\int\sum_{j=m_1}^{m_2}\chi_{Q^*_j}\vec F^*(x)^{p_1}dx$ can be arbitrary small if both $m_1$ and $m_2$ are large.
 Therefore $ \{\sum_{1\le j\le m}b_{j}^{(k)}\}_m$ is Cauchy in $H^{p_1}(\rn,\ell^2(L))$. Since $H^{p_1}(\rn,\ell^2(L))$ is complete, 
the limit of the sequence $\{b^{(k)}\}_{k=1}^L$ exists and 
$\|\{b^{(k)}\}_k\|^{p_1}_{H^{p_1}(\rn,\ell^2(L))}\le C\al^{p_1-p}\|\vec F^*\|_{\wlp}^p.$ This  $C$ is independent of $L$.

Moreover, since for each $k$, $\sum_{j=1}^mb_j^{(k)}\rar b^{(k)}$ in $\mcs'$ as $m\rar\iy$, we have 
\begin{align*}
\sup_{t>0}(\sum_{k=1}^L|b^{(k)}*\psi_t|^2)^{\f12}\le&\sum_{j=1}^{\iy}\sup_{t>0}(\sum_{k=1}^L|(b_j^{(k)}*\psi_t)(x)|^2)^{\f12}\\
\le&\sum_{j=1}^{\iy}\vec F^*\chi_{Q_j^*}+C\al\sum_{j=1}^{\iy}(\tf{l_j}{|x-x_j|})^{N+n+1}\chi_{(Q_j^*)^c}\\
\le& C\vec F^*\chi_{\Om_\al}+C\al\sum_{j=1}^{\iy}(\tf{l_j}{l_j+|x-x_j|})^{N+n+1},
\end{align*}
from which we have 
\begin{align*}
&\la^p|\{\sup_{t>0}(\sum_{k=1}^L|b^{(k)}*\psi_t|^2)^{\f12}>\la\}|\\
\le&\la^p|\{\vec F^*>\f{\la}2\}|+\la^p|\{C\al\sum_j(\tf{l_j}{|x-x_j|+l_j})^{N+1}>\f{\la}2\}|\\
\le &C\|\vec F\|_{\hi(\rn,\ell^2)}^p+C\al^p\int_{\rn}\sum_j(\tf{l_j}{|x-x_j|+l_j})^{(N+1)p}dx\\
\le&C\|\vec F\|_{\hi(\rn,\ell^2)}^p+C\al^p\sum_j|Q_j|\\
\le &C\|\vec F\|_{\hi(\rn,\ell^2)}^p\, . 
\end{align*}

We  can therefore define $\vec G=\{g^{(k)}\}_{k=1}^L$ as $g^{(k)}=f_k-\sum_jb_j^{(k)}$ and obviously $\vec G$ lies in $\hi(\rn,\ell(L))$.
To estimate $ (\sum_{k=1}^L|(g^{(k)}*\psi_t)(x)|^2)^{\f12}$, let's consider first the case $x\notin\Om_\al$. Then 

\begin{align*}
\mathcal{M}_0(\vec G)(x)=\,&\sup_{t>0}(\sum_{k=1}^L|(g^{(k)}*\psi_t)(x)|^2)^{\f12}\\
\le &\,\sup_{t>0}(\sum_{k=1}^L|(f_k*\psi_t)(x)|^2)^{\f12}+\sup_{t>0}(\sum_{k=1}^L|(\sum_jb_j^{(k)}*\psi_t)(x)|^2)^{\f12}\\
\le&\,C\vec F^*(x)\chi_{\Om_\al^c}(x)+\sum_{j=1}^{\iy}C\al(\tf{l_j}{l_j+|x-x_j|})^{N+n+1}.
\end{align*}

We claim that this estimate is true for almost all $x$. 
\newcommand{\mr}{\mathrm}

Now let's consider the case $x\in\Om_\al$. There exists some $m$ such that $x\in Q_m$, and we can divide $\mathbf{N}$ into two sets $\mr I$ and $\mr {II}$ with $j\in \mr I$ 
if $Q_j^*\cap Q_m^*\neq\emptyset$ and $j\in \mathrm{II}$ otherwise. 
$$
\mathcal{M}_0(\{g^{(k)}\})(x)\le \mathcal{M}_0(\{f_k-\sum_{j\in \mr I}b^{(k)}_j\})(x)+\mathcal{M}_0(\{\sum_{j\in \mathrm{II}} b_j^{(k)}\})(x)
$$
Since 
$x\notin Q_j^*$ for $j\in\mr{II},$
$\mathcal{M}_0(\{\sum_{j\in \mathrm{II}} b_j^{(k)}\})(x)
\le \sum_{j\in\mr{II}}
C\al (\tf{l_j}{l_j+|x-x_j|})^{N+n+1}.$

We notice that 
\begin{align*}
 \mathcal{M}_0(\{f_k-\sum_{j\in \mr I}b^{(k)}_j\})(x) 
\le\,\mathcal{M}_0(\{f_k-\sum_{j\in \mr I}f_k\vp_j\})(x)+\mathcal{M}_0(\{\sum_{j\in \mr I}P_j^{(k)}\vp_j\})(x).
\end{align*}
To estimate the second term, we have 
$$
\Big(\sum_k\big|\sum_{j\in\mr I}P_j^{(k)}\vp_j\big|^2\Big)^{\f12}\le \sum_{j\in\mr I}\big(\sum_k | P_j^{(k)}|^2\big)^{\f12}\vp_j\le C\al,
$$
and then 
$
\mathcal{M}_0(\{\sum_{j\in \mr I}P_j^{(k)}\vp_j\})(x)\le C\al.
$

To estimate the other term, we notice that we need only to consider the case $t>cl_m$ ($c$ is independent of $m$), 
otherwise $  (\sum_k((f_k-\sum_{j\in\mr I}f_k\vp_j)*\psi_t)^2)^{\f12}(x)=0$ since $\psi$ is supported 
in $B(0,1)$, $ f_k(1-\sum_{j\in\mr I}\vp_j)$ is supported outside $Q_m^{**}$ and $x\in Q_m$. 
If $t<10\sqrt nl_m$, then
\begin{align*}
 \Big(\sum_k |(f_k-\sum_{j\in\mr I}f_k\vp_j)*\psi_t |^2\Big)^{\f12}(x) 
=&\,(\sum_k| \langle f_k,\Phi \rangle|^2)^{\f12}\\
\le&\,\mathfrak{N}_N(\Phi;x,10\sqrt nl_m)\inf_{|z-x|\le 10\sqrt nl_m}
\vec F^*(z)\\
\le&\, C\al\\
\le&\, C\al(\tf{l_m}{l_m+|x-x_m|})^{N+n+1}, 
\end{align*}
where $\Phi(y)=\psi_t(x-y)(1-\sum_{j\in\mr I}\vp_j(y))$.  

For $t>10\sqrt nl_m$, $\Phi(y)=\psi_t(x-y)(1-\sum_{j\in\mr I}\vp_j(y))=\psi_t(x-y)$ since the support of $ \sum_{j\in\mr I}\vp_j$ is contained in $ B(x,9\sqrt nl_m)$.
We can check that   $\mathfrak{N}_N(\Phi;x,t)\le C$ with $C$ independent of $x$ and $t$.
Therefore
 $$
 \Big(\sum_k |(f_k-\sum_{j\in\mr I}f_k\vp_j)*\psi_t |^2\Big)^{\f12}(x) \le\mathfrak{N}_N(\Phi;x,t)\inf_{|z-x|\le t}
 \vec F^*(z)\le C\al\, .
$$

To summarize, we have showed that 
$$\mathcal{M}_0(\vec G)(x)\le C
\vec F^*(x)\chi_{\Om_\al^c}(x)+\sum_{j=1}^{\iy}C\al(\tf{l_j}{l_j+|x-x_j|})^{N+n+1}\ a.e.$$
This gives us that $\|\mathcal{M}_0(\vec G)\|_{L^{p_2}}\le C\al^{\f{p_2-p}{p_2}}\|\vec F\|_{H^{p,\iy}}^{\f{p}{p_2}}$ since 
$
\|\sum_{j=1}^{\iy}(\tf{l_j}{l_j+|x-x_j|})^{N+n+1}\|_{L^{p_2}}
\le C|\Om|^{\tf1{p_2}}.
$

\end{proof}


We have the following corollary.

\begin{corollary}\label{SI}
Let $0<p<\iy$ and suppose that 
 $\{K_j(x)\}_{j=1}^L$ is a family of kernels defined on $\rn\backslash\{0\}$ satisfying
$$
\sum_{j=1}^L |\p^{\al}K_j(x)|\le A|x|^{-n-|\al|}<\iy  
$$
for all $|\al|\le \max\{[n/p]+2,n+1\}$ and 
$$
\sup_{\xi\in\rn} \sum_{j=1}^L |  \wh{K_j}(\xi)|\le B<\iy\, .
$$
Then for some $0<p$ there exists a constant $C_{n,p}$ independent of $L$ such that 
\begin{equation}\lab{963738}
\Big\|\sum_{j=1}^LK_j*f_j\Big\|_{\hi(\rn)}\le C_{n,p}(A+B)\|\{f_j\}_{j=1}^L\|_{\hi(\rn,\ell^2(L))}.
\end{equation}
\end{corollary}

\begin{proof}

We pick $p_1<p<p_2$ such that $p_1\le 1$ and $[n/p_1]+1=\max\{[n/p]+2,n+1\}$. Then \eqref{963738} holds with 
$H^{p,\nf}$ replaced by both $H^{p_1}$ and $H^{p_2}$ in view of Theorem 6.4.14 in \cite{G2}. 
Using Theorem \ref{In1} we derive the required conclusion. 
\end{proof}

\section{Square function characterization of   $H^{p,\nf}$}

We discuss an important characterization  of Hardy spaces 
in terms of   Littlewood--Paley  square functions. The vector-valued Hardy spaces 
and the action of singular integrals on them  
are crucial tools in obtaining this characterization.

We first      set up the  notation. 
We fix a radial Schwartz function $\Psi$ on $\rn$ whose Fourier 
transform is nonnegative, supported in the annulus $1-\f17\le |\xi|\le 2$,   and    satisfies 
\begin{equation*}
\sum_{j\in\zzz} \wh{\Psi}(2^{-j}\xi)=1
\end{equation*} 
for all $\xi\neq 0$. Associated with this bump, we define the 
Littlewood--Paley operators 
$\De_j$  given by multiplication  on 
the Fourier transform side 
by the function $\wh{\Psi}(2^{-j}\xi)$, that is, 
\begin{equation}\lab{6.4.defDej}
\De_j (f)  =  \Psi_{2^{-j}}*f\, . 
\end{equation} 
We also define the function $\Phi$ by $\wh\Phi(\xi)=\sum_{j\le0}\wh\Psi(2^{-j}\xi)$ for $\xi\neq 0$ and $\wh\Phi(0)=1$.
Now we're going to prove Theorem \ref{sq}.
\begin{proof}[Proof of Theorem \ref{sq}]
Choose $f\in \hi$
and denote $f_M= \sum_{|j|\le M}\De_j(f)=\Phi_{2^{-M}}*f-\Phi_{2^M}*f$ and $  S(f)=(\sum_{|j|\le M}|\De_j(f)|^2)^{\f12}$.
We can see that $S(f+g)\le S(f)+S(g)$ and $S(af)=|a|S(f)$. We also know from \cite{G2} 
that $S$ maps $H^{p_i} $ to $L^{p_i}$ ($i=1,2$) bounded by the 
square function characterization of Hardy spaces.  Then by Theorem \ref{In1} it follows that $S$ maps $\hi$ to $L^{p,\iy}$ bounded for 
$p\in(p_1,p_2)$,
so 
$$
\|S(f)\|_{L^{p,\iy}}\le C\|f\|_{\hi}.
$$

Applying Fatou's lemma for $\wlp$ spaces we have 
$$
\|(\sum_{j\in \zzz}|\De_j(f)|^2)^{\f12}\|_{\wlp}\le\liminf_{M\rar\infty}\|(\sum_{|j|\le M}|\De_j(f)|^2)^{\f12}\|_{\wlp}\le C\|f\|_{\hi}.
$$

Assume   that we have a distribution 
$f\in\mcs'$ such that $ \|(\sum_{j\in \zzz}|\De_j(f)|^2)^{\f12}\|_{\wlp}<\iy$ and define $f_j=\De_j(f)=\Psi_{2^{-j}}*f$. We can show that 
$\{f_j\}_{j\in\zzz}\in\hi(\rn,\ell^2)$.

To prove this, let's take $\vp\in\mcs$ whose Fourier transform takes value $0$ for $|\xi|\ge 2$ and  $1$ for $|\xi|\le1$. 
$\wh{\vp_t*\De_j(f)}(\xi)=\wh\vp(t\xi)\wh\Psi(2^{-j}\xi)\wh f(\xi),$
so it's just $\wh{\De_j*f}$ if $\f1t>2^{j+1}$ and $0$ if $\f1t<\f37\cdot2^j$. Therefore 
$$
\sup_{t>0}|\vp_t*\De_j(f)|\le|\De_j(f)|+\sup_{\f73\cdot2^{-j}\ge t\ge 2^{-(j+1)}}|\vp_t*\De_j(f)|.
$$
For $\f73\cdot2^{-j}\ge t\ge 2^{-(j+1)}$, by lemma 6.5.3 of \cite{G2}
\begin{align*}
|\vp_t*\De_j(f)|(x)
\le &C_N''M(|\De_j(f)|^r)^{\f1r}(x),
\end{align*}
where $M$ is the Hardy-Littlewood maximal function and $r<\min(2,p)$. Apply Proposition \ref{mf} to obtain  
\begin{align*}
&\|\sup_{t>0}(\sum_{j\in\zzz}|\vp_t*\De_j(f)|^2)^{\f12}\|_{\wlp}\\
\le&\,\|(\sum_{j\in\zzz}(\sup_{t>0}|\vp_t*\De_j(f)|)^2)^{\f12}\|_{\wlp}\\
\le &\,C_p \|(\sum_{j\in\zzz}(|\De_j(f)|)^2)^{\f12}\|_{\wlp}+C_p\|(\sum_{j\in\zzz}(M(|\De_j(f)|^r)^{\f1r})^2)^{\f12}\|_{\wlp} \\
\le &\,C_p'\|(\sum_{j\in\zzz}(|\De_j(f)|)^2)^{\f12}\|_{\wlp}. 
\end{align*}
The fact that 
$\|\sum_{k=1}^{\iy}g_k\|_{\wlp}<\iy$ doesn't imply that $\{\sum_{k=1}^Mg_k\}_M$ is a Cauchy sequence 
in $\wlp$, so we cannot apply the method used in $H^p$ case. 
But we still can use a new method which is also applicable to the $H^p$ case.

Let $\wh\eta(\xi)=\wh\Psi(\xi/2)+\wh\Psi(\xi)+\wh\Psi(2\xi)$, then by Corollary \ref{SI}
\begin{align*}
\big\|\sum_{|j|\le M}\De_j(f)\big\|_{\hi}=\,&\,\|\sum_{|j|\le M}\De^{\eta}_j(f_j)\|_{\hi}\\
\le&\, C\|\{f_j\}\|_{\hi(\rn,\ell^2(M))}\\
\le&\,C\|\sup_{t>0}(\sum_{|j|\le M}|\vp_t*f_j|^2)^{\f12}\|_{\wlp}\\ 
\le &\,C\|(\sum_{|j|\le M}\sup_{t>0}|\vp_t*f_j|^2)^{\f12}\|_{\wlp}\\ 
\le &\,C\|(\sum_{j\in\zzz}(|\De_j(f)|)^2)^{\f12}\|_{\wlp} .
\end{align*}

So $\{\sum_{|j|\le M}\De_j(f)\}_M$ is a bounded sequence in $\hi$ uniformly in $M$ and we are able to use the following lemma.
\begin{lemma}
If $\{f_j\}$ is bounded by $B$ in $\hi$ (or $H^p$), then there exists a subsequence $\{f_{j_k}\}$ such that $f_{j_k}\rar f$ in $\mcs'$ for some $f$ in $\hi$ (or $H^p$) with 
$\|f\|_{\hi}\le B$ (or $\|f\|_{H^p}\le B$). 
\end{lemma}
\begin{proof}
$V=\{\vp\in\mcs:\mathfrak{N}_N (\vp)<\f1{2BC}\}$ is a neighborhood of $0$ in $\mcs$ and 
$$
|\langle f_j,\vp \rangle |\le C\ \mathfrak{N}_N (\vp)\|f_j\|_{\hi}\le 1.
$$ 
So by the separability of $\mcs$ we have the weak*-compactness of this sequence, which means that there exists a subsequence $\{f_{j_k}\}$ such that 
$f_{j_k}\rar f$ in $\mcs'$. Therefore $\dis \|f\|_{\hi}\le \liminf_{k\rar\iy}\|f_{j_k}\|_{\hi}\le B$.
\end{proof}

By the lemma we know that $ \sum_{|j|\le M_k}\De_j(f)\rar g$ in $\mcs'$ with 
$$
\dis\|g\|_{\hi}\le C\Big\|\big(\sum_{j\in\zzz}|\De_j(f)|^2\big)^{\f12}\Big\|_{\wlp}.
$$
Moreover, we know that $ \sum_{|j|\le M}\De_j(f)\rar f$ in $\mcs'/\mathcal{P}$, so there is a unique polynomial $Q$ such that $\dis g=f-Q.$
\end{proof}

\begin{corollary}\lab{last}
Let $\Psi$ be a  smooth bump whose Fourier transform is
supported in an  annulus that does not contain the origin   and
satisfies for some positive integer $q$:
$$
\sum_{j\in \zzz} \wh{\Psi}(2^{-jq}\xi) = 1, \qq \xi\in \mathbb
R^d\setminus\{0\}.
$$
Then for any $0< p<\nf$ there is a constant $C(p,n)$ (that also
depends on $\Psi$) such that for all  functions $f$ in $L^r$ with some $r\in[1,\nf]$ and whose ``lacunary" square function $
\mathbf S^\Psi_q(f) = (\sum_{\ell\in \zzz } |\De_{q\ell}^\Psi
(f)|^2)^{1/2}\, 
$ lies in weak $\lp$ we
have
$$
\| f\|_{L^{p,\nf}}\le C(p,n) \big\|\mathbf S^{\Psi}_q(f)
\big\|_{L^{p,\nf }} \, .
$$

\end{corollary}

\begin{proof}
Let us prove the case that $\Psi$ satisfies the assumptions of Theorem \ref{sq} and therefore $q=1$.
Since $f\in  L^r$, it is an element of $\mathcal S'(\rn)$. The square function of $f$ lies in weak $\lp$, thus Theorem \ref{sq} yields the existence of a polynomial $Q$ such that $f-Q$ lies in $H^{p,\nf}$. 
By the Lebesgue differentiation theorem it follows that for almost all $x\in \rn$ we have
\begin{equation}\label{LWLKIY}
|f(x)-Q(x)|\le C  \sup_{t>0}|(\vp_t*(f-Q))(x)|\, , 
\end{equation}
where $\vp$ is a smooth compactly supported function with $\intrn \vp(x)\,dx =1$. Taking  $\wlp$ norms in both sides of \eqref{LWLKIY}, and using Theorem \ref{sq}  we obtain  that 
$$
\|f-Q\|_{\wlp} \le C' 
 \big\| \Big( \sum_{j\in \zzz}  |\De_j (f-Q )|^2\Big)^{\f12}\big\|_{L^{p,\nf}}
 = C'\big\| \Big( \sum_{j\in \zzz}  |\De_j (f  )|^2\Big)^{\f12}\big\|_{L^{p,\nf}}\, . 
 $$

If $f\in L^r$ and $g=f-Q\in\wlp$, then choose $p_2>p$ and denote $m=\max(r,p_2)$. By Lemma \ref{LD} we will have
\begin{align*}
&\lim_{M\rar\iy}\f{|\{1<x<M:\,\,|Q|>x^{-\f1m}\}|}{M}\\
\le &\lim_{M\rar\iy}\f{|\{1<x<M:\,\,2|g|>x^{-\f1{p_2}}\}|}{M}+\lim_{M\rar\iy}\f{|\{1<x<M:\,\,2|f|>x^{-\f1r}\}|}{M}
=0
\end{align*}
which implies that $Q=0$ since $x^{-\tf1m}\rar0$ as $x\rar\iy$.

For more general case, the support of $\wh{\Psi}(\xi)$ may intersect more supports of functions of the form $\wh{\Psi}(2^{-jq}\xi)$ and the number of 
intersection is finite since
the support of $\wh{\Psi}$ is a compact annulus that does not contain $0$. If we take $\vp$ as in Theorem \ref{sq}, then
$
\sup_{t>0}|\vp_t*\De_{jq}(f)|\le|\De_{jq}(f)|+\sup_{a\cdot2^{-jq}\ge t\ge b\cdot2^{-jq}}|\vp_t*\De_{jq}(f)|,
$
where $a$ and $b$ are constants depending on the support of $\wh{\Psi}$.
If we choose an appropriate $\eta$ satisfying that
 $\wh{\eta}(\xi)=1$ on the support of $\wh{\Psi}$, then
 there is no difficulty to apply Corollary \ref{SI} to show that 
 \begin{align*}
 \big\|\sum_{|j|\le M}\De_{jq}(f)\big\|_{\hi}=&\,\,\|\sum_{|j|\le M}\De^{\eta}_{jq}(f_j)\|_{\hi}\\
\le&\,C\|\sup_{t>0}(\sum_{|j|\le M}|\vp_t*f_j|^2)^{\f12}\|_{\wlp}\\
\le &\,C'\|(\sum_{j\in\zzz}(|\De_{jq}(f)|)^2)^{\f12}\|_{\wlp},
 \end{align*}
 which gives that 
$
\| f-Q\|_{H^{p,\nf}}\le C(p,n) \|\mathbf S^{\Psi}_q(f)
\|_{L^{p,\nf }} \, .
$
The rest discussion then follows easily as we did in the case that $\Psi$ satisfies assumptions of Theorem \ref{sq}. 

\end{proof}
The preceding corollary has applications in the theory of paraproducts. See \cite{GKM}. Moreover, the following corollary 
can be proved similarly to the previous corollary.
\begin{corollary}\lab{c6.4.16}
Fix $\Psi$ in $\cs(\rn)$ with Fourier transform
supported in  $ \f67 \le |\xi| \le 2$, equal  $1$ on the $1\le |\xi| \le \f{12}{7}$, and satisfy
$\sum_{j\in \zzz} \wh{\Psi}(2^{-j}\xi) = 1$ for $\xi\neq 0$. Fix $b_1,b_2$ with $b_1<b_2$ and define
a Schwartz function $\Om$ via
$\wh{\Om}(\xi) = \sum_{j=b_1}^{b_2} \wh{\Psi}( 2^{-j} \xi ).
$
Define
$\De_k^\Om(g)\,\sphat\,\, (\xi) = \wh{g}(\xi)\wh{\Om} (2^{-k}\xi)$, $k\in \zzz$.   Let $q=b_2-b_1+1$,  $0<p\le 1$, and
fix $r\in \{0,1,\dots , q-1\}$.
Then there exists a constant
$C=C_{n, p,b_1,b_2,\Psi}$ such that for all $f\in \hi(\rn)$ we have
\begin{equation}\lab{6.4.6.1st-dire-0}
\Big\|\Big(\sum_{j= r\!\!\!\!\! \mod q} |\De_j^{\Om}(f)|^2\Big)^{\!\f12} \Big\|_{\wlp}\le
C \big\|f\big\|_{\hi}\, .
\end{equation}
Conversely, suppose that a  tempered distribution $f$   satisfies
\begin{equation}\lab{6.4.6.2nd-dire099887-0}
\Big\|\Big(\sum_{j= r\!\!\!\!\! \mod q} |\De_j^{\Om}(f)|^2\Big)^{\!\f12} \Big\|_{\wlp}<\nf \, .
\end{equation}
Then there exists a unique polynomial $Q(x)$ such that $f-Q$
lies in the weak Hardy space
 $\hi$ and satisfies for some constant
$C=C_{n, p,b_1,b_2,\Psi}$
\begin{equation}\lab{6.4.6.2nd-dire-0}
\f{1}{C}\big\|f-Q\big\|_{\hi}\le
\Big\|\Big(\sum_{j= r\!\!\!\!\! \mod q} |\De_j^{\Om}(f)|^2\Big)^{\!\f12} \Big\|_{\wlp} \, .
\end{equation}
\end{corollary}
\begin{proof}
We have proved the direction that $\f{1}{C}\|f-Q\|_{\hi}\le(\sum_{j= r\!\!\mod q} |\De_j^{\Om}(f)|^2)^{\!\f12} \|_{\wlp} \, .$ 
$(\sum_{j= r\! \mod q} |\De_j^{\Om}(f)|^2)^{1/2}\le\sum_{k=1}^q(\sum_j|\De_{qj+k}^{\Psi}(f)|^2)^{1/2}\le q(\sum_j|\De_j^{\Psi}(f)|^2)^{1/2}$ comes from 
the fact that
$\wh{\Om}(\xi) = \sum_{j=b_1}^{b_2} \wh{\Psi}( 2^{-j} \xi )$, which proves the other direction.
\end{proof}

\smallskip
 
I want to express my deepest gratitude  to Professor  L. Grafakos, who  
gave me a lot of valuable suggestions. Without him
I could not have  finished this article.


\end{document}